\documentclass{amsart}
\usepackage[margin=1.5cm]{geometry}
\numberwithin{equation}{section}
\usepackage{amssymb}
\usepackage{amsmath, amsfonts,amsthm,amssymb,amscd, verbatim,graphicx,color,multirow,booktabs, caption,tikz,tikz-cd, mathdots,bm}
\usepackage{tikz-cd}
 \usepackage{hyperref}
\usetikzlibrary{positioning}

\newtheorem{theorem}{Theorem}
\newtheorem{lemma}[theorem]{Lemma}

  \newtheorem{conjecture}[theorem]{Conjecture}

      \theoremstyle{definition}

     \theoremstyle{remark}
     \newtheorem{remark}[theorem]{Remark}

\newcommand{\Sym}{\mathop{\mathrm{Sym}}}
\newcommand{\Alt}{\mathop{\mathrm{Alt}}}
\newcommand{\Aut}{\mathop{\mathrm{Aut}}}
\newcommand{\Soc}{\mathop{\mathrm{Soc}}}
\newcommand{\Sol}{\mathop{\mathrm{Sol}}}

 \definecolor{mycolor}{rgb}{0.55,0.0,0.16}
  \definecolor{myred}{rgb}{0.6,0.0,0.16}
  \definecolor{mygreen}{rgb}{0.0,0.6,0.16}
  \definecolor{myviolet}{rgb}{1,0,1}

\hypersetup{
colorlinks=false,
allbordercolors=myred,
citebordercolor=mygreen
}

\makeatletter
\@namedef{subjclassname@2020}{%
  \textup{2020} Mathematics Subject Classification}
\makeatother

\begin{document}
\title[Nilpotent subgroups of class $2$ in finite groups]{Nilpotent subgroups of class $2$ in finite groups}
\author[L. Sabatini]{Luca Sabatini}
\address{Luca Sabatini, Dipartimento di Matematica  e Informatica ``Ulisse Dini'',\newline
 University of Firenze, Viale Morgagni 67/a, 50134 Firenze, Italy} 
\email{luca.sabatini@unifi.it}
% \keywords{...}        
	\maketitle

	\begin{center} 
 {\itshape \small In memory of Carlo Casolo (1958-2020)}
	\end{center}

       \begin{abstract}
We show that every finite group $G$ of size at least $3$ has a nilpotent subgroup of class at most $2$
and size at least $|G|^{1/32\log\log|G|}$.
This answers a question of Pyber, and is essentially best possible.
\end{abstract}
          
         \vspace{0.1cm}
   \section{Introduction}
   
While dealing with a finite group,
it is a question of great interest whether it has some large abelian/nilpotent/solvable subgroup.
In 1976, Erd\H{o}s and Straus \cite{1976ES} proved that every finite group $G$ has an abelian subgroup of size roughly $\log|G|$.
In \cite{1997Pyber},
Pyber showed that this can be improved to $2^{\varepsilon \sqrt{\log|G|}}$,
for some fixed $\varepsilon>0$.
In the same article,
a key proposition shows that every finite group $G$ has a solvable subgroup of size at least $|G|^{\varepsilon/\log\log|G|}$
for some fixed $\varepsilon>0$, and the inequality is sharp.
In this paper, we prove that actually every finite group has a such large nilpotent subgroup of class at most $2$.

\begin{theorem} \label{thMain}
Every finite group $G$ of size at least $3$
has a nilpotent subgroup of class at most $2$ and size at least $|G|^{1/32\log\log|G|}$.
\end{theorem}
  \vspace{0.1cm}

By a classical result of Dixon \cite{1967Dixon},
a solvable subgroup of the symmetric group on $n$ elements has size at most $24^{(n-1)/3}$.
This shows that the ``$\log\log$'' factor cannot be replaced by anything smaller.\\
All the logarithms in this paper are to base $2$.

\section{Proof of Theorem \ref{thMain}}

The key tool to find nilpotent subgroups of class at most $2$
is the following generalization of a remark of Thompson \cite{1969Thompson}.
We remind the reader that a finite group lies in this class if and only if its derived subgroup is contained in its center.

\begin{lemma} \label{lemmaNilp2}
    Let $G$ be a finite group,
    and let $H \leqslant G$ be of minimal size among the subgroups of $G$
    which provide an abelian section of maximal order.
    Then $H$ is nilpotent of class at most $2$.
    \end{lemma}
    \begin{proof}
     We first show that $H$ is nilpotent.
    If not, take $M \leqslant H$ a non-normal maximal subgroup of $H$.
    Then $MH'=H$ or $MH'=M$, but the second is impossible because $M$ is not normal in $H$.
    It follows that
    $$ |H/H'| \> = \> |MH'/H'| \> = \> |M/M \cap H'| \> \leq \> |M/M'| . $$
    Since $|M|<|H|$, then one can conclude that $H$ is nilpotent,
    i.e. it is the direct product of its Sylow's subgroups $P_1,...,P_k$.
   We show that each of these has nilpotency class at most $2$.
   Since both the derived subgroup and the center factorize through a direct product, the proof shall follow.
   Now $H$ has the property that $H/H'$
   is larger than every other abelian section of $H$.
   We see that $P_i$ has the same property, for every $1 \leq i \leq k$.
   In fact, if not, then since $H/H' \cong P_1/(P_1)' \times ... \times P_k/(P_k)'$,
   an abelian section of $H$ of size larger than $H/H'$ would be found.
   Then, the proof follows from \cite[Lemma]{1969Thompson}.
    \end{proof}
      \vspace{0.1cm}
    
    A finite nilpotent group $G$ of class at most $2$ has an abelian section of size at least $|G|^{1/2}$.
    Thus, a large nilpotent subgroup of class at most $2$ and a large abelian section are essentially equivalent.
    At the end of \cite{1997Pyber} Pyber remarks that every finite group $G$ has an abelian section of size at least
    $|G|^{\varepsilon/(\log\log|G|)^2}$ for some fixed $\varepsilon>0$, and wonders whether much more is true.
    Of course, Theorem \ref{thMain} fullfils this purpose by showing that one ``$\log\log$'' factor can be dispensed with.\\
    
   We say that a finite group is of {\bfseries alternating type}
    if all its nonabelian composition factors are alternating of degree at least $23$
    (alternating factors of smaller degree are excluded for technical reasons).
    We report the following useful result.

\begin{lemma}[Corollary 2.3 (b) in \cite{1997Pyber}] \label{lemPyb}
Every finite group $G$ contains an alternating type subgroup of size at least $|G|^{1/5}$.
\end{lemma}
  \vspace{0.1cm}

Pyber proves Lemma \ref{lemPyb} with $8$ in place of $23$, but his argument can be easily extended up to this threshold.
Let us spend a few words explaining how this is obtained.
Looking in a chief series $1=G_0 \lhd G_1 \lhd ... \lhd G_r = G$,
we remove the factors $G_i/G_{i-1}$ which are direct products of large alternating groups.
It is always possible, inside each of the remaining factors, to find a very large intravariant solvable subgroup
(remark that $\Alt(22)$, but not $\Alt(23)$, contains a $3$-subgroup which is sufficiently large).
Then one can merge these solvable sections and the full alternating factors,
using a quite technical result of Chunikhin.\\

Let $(a_i)_{i \geq 1}$ be a sequence of distinct integers greater than $6$, and let $(b_i)_{i \geq 1}$ be a sequence of positive integers.
Let $H := \prod_i \Alt(a_i)^{b_i}$.
The next is a simple remark.

\begin{lemma} \label{lemAut}
$\Aut(H) \> \cong \> \prod_i \Sym(a_i) \wr_{b_i} \Sym(b_i)$.
\end{lemma}
\begin{proof}
For every $i$, we have that $\Alt(a_i)^{b_i}$ is a characteristic subgroup of $H$.
It follows that $\Aut(H) \cong \prod_i \Aut(\Alt(a_i)^{b_i})$.
Since $a_i \geq 7$, it is easy to see that $\Aut(\Alt(a_i)^{b_i}) \cong \Sym(a_i) \wr_{b_i} \Sym(b_i)$ for every $i$, as desired.
\end{proof}
  \vspace{0.1cm}

We now prove Theorem \ref{thMain} in the range of solvable groups.

\begin{lemma} \label{lemmaSolv}
Every finite solvable group $G$ of size at least $3$ contains an abelian section of size at least $|G|^{1/4 \log\log|G|}$.
Therefore it contains a such large nilpotent subgroup of class at most $2$.
\end{lemma}
\begin{proof}
From \cite[Corollary]{1991Heineken} we have that $G$
    contains a nilpotent subgroup $H$ of size at least $|G|^{1/3}$.
 The classical result \cite[Theorem 2.54]{1933Hall} says that a finite $p$-group of nilpotency class $c \geq 1$
 has derived length at most $1+\log c$, and this is easily extended to every finite nilpotent group.
 So if $c$ is the nilpotency class of $H$, then $H$ has derived length at most $1+\log c \leq 1+\log\log|H|$.
    Using pigeonhole on the derived series, we find an abelian section of size at least
    $$ |H|^{1/(1+\log\log|H|)} \> \geq \>  |G|^{1/(3+3\log(3^{-1}\log|G|))} \> \geq \>  |G|^{1/4\log\log|G|} , $$
    where we used that $x^{1/(1+\log \log x)}$ is an increasing function when $x \geq 3$.\\
    The second part follows from Lemma \ref{lemmaNilp2}.
\end{proof}
  \vspace{0.1cm}

It is worth noticing that, unlike Theorem \ref{thMain}, perhaps Lemma \ref{lemmaSolv} could be improved significantly.
The next conjecture is really equivalent to \cite[Problem 14.76]{2018Kourovka}.

\begin{conjecture} \label{conjPyber}
There exists an absolute constant $\varepsilon>0$ such that
every finite solvable group $G$ has a nilpotent subgroup of class at most $2$ and size at least $|G|^\varepsilon$. 
\end{conjecture}

\begin{proof}[Proof of Theorem \ref{thMain}]
Up to a factor of $1/5$, we can use Lemma \ref{lemPyb} to replace an arbitrary finite group with a group of alternating type.
Moreover, for groups of cardinality less than $20$, the desired subgroup can be found trivially.
So let $G$ be a group of alternating type and size at least $20$.
Let us denote by $\Sol(G)$ the solvable radical of $G$, i.e. the largest solvable normal subgroup.
First suppose that $|\Sol(G)| > |G|^{12/19}$.
Thus, from Lemma \ref{lemmaSolv}, $\Sol(G)$ and then $G$ has an abelian section of size at least
\begin{equation} \label{eq1}
|\Sol(G)|^{1/4\log\log|\Sol(G)|} \> \geq \> |G|^{3/19\log\log|G|}
\end{equation}
and so we are done.
Otherwise, if $|\Sol(G)| \leq |G|^{12/19}$, then $\widetilde{G} := G/\Sol(G)$ has size at least $|G|^{7/19}$.
Moreover, $\widetilde{G}$ has no non-trivial solvable normal subgroups.
Let $H:=\Soc(\widetilde{G})$ be the subgroup generated by the minimal normal sugroups.
Since $\widetilde{G}$ has alternating type, we have $H \cong \prod_i \Alt(a_i)^{b_i}$ for some $a_i \geq 23$ and $b_i \geq 1$.
It is well known that $C_{\widetilde{G}}(H) \leqslant H$, and so $C_{\widetilde{G}}(H) = Z(H)=1$.
It follows that $\widetilde{G}$ has an embedding into $\Aut(H)$.
Then, Lemma \ref{lemAut} provides $|\widetilde{G}| \leq \prod_i (a_i!)^{b_i} b_i!$,
and $\log|\widetilde{G}| \leq \sum_i b_i (a_i \log a_i + \log b_i)$.
On the other hand $\widetilde{G}$ contains a copy of $\prod_i \Alt(a_i)^{b_i}$.
For every $i$, $\Alt(a_i)$ contains a $3$-elementary abelian subgroup of size $3^{\lfloor a_i/3 \rfloor} \geq 2^{3a_i/7}$.
This shows that $H$ and then $\widetilde{G}$ contains an abelian subgroup of size at least $\prod_i 2^{3a_i b_i/7}$.
The logarithm of this quantity is $\tfrac{3}{7} \sum_i a_i b_i$.
For every $i$ we have
$$ a_i b_i \> \geq \> \frac{b_i(a_i \log a_i +\log b_i)}{\log b_i +\log a_i} \> \geq \> 
\frac{b_i(a_i \log a_i +\log b_i)}{\log(b_i(a_i \log a_i +\log b_i))} . $$
It follows that
\begin{align*}
\sum_i a_i b_i & \> \geq \> 
\frac{\sum_i b_i(a_i \log a_i +\log b_i)}{\max_i (\log(b_i(a_i \log a_i +\log b_i)))} \\ & \> \geq \> 
\frac{\sum_i b_i(a_i \log a_i +\log b_i)}{\log( \sum_i (b_i(a_i \log a_i +\log b_i)))} \\ & \> \geq \> 
\frac{\log|\widetilde{G}|}{\log \log|\widetilde{G}|} .
\end{align*}
To sum up, we proved that $\widetilde{G}$ and then $G$ contains an abelian section of size at least
\begin{equation} \label{eq2}
|\widetilde{G}|^{3/7\log\log|\widetilde{G}|} \> \geq \>  |G|^{3/19\log\log|G|} , 
\end{equation} 
and the proof is complete via Lemma \ref{lemmaNilp2},
once we put back the factor $1/5$ in (\ref{eq1}) and (\ref{eq2}).
\end{proof}

\begin{remark}
We would like to make a remark about the gap between
Pyber's $2^{\varepsilon \sqrt{\log|G|}}$ bound and the much larger $|G|^{1/32 \log\log|G|}$.
This difference is explained by the ``distance'' between nilpotent groups of class $2$ and abelian groups
(in contrast to the ``proximity'' between solvable groups and nilpotent groups of class $2$, for example).
By a non-explicit construction of Ol'shanskii \cite{1978O},
there exists an arbitrarily large finite group $G$, nilpotent of class $2$, without abelian subgroups larger than $2^{\alpha \sqrt{\log|G|}}$
for some fixed $\alpha>0$.
\end{remark}

\vspace{0.3cm}
\bibliographystyle{amsplain}

\thebibliography{10}




\bibitem{1967Dixon} J.D. Dixon, \textit{The Fitting subgroup of a linear solvable group},
Journal of the Australian Mathematical Society \textbf{7} (1967), 417-424.

\bibitem{1976ES} P. Erd\H{o}s, E.G. Straus, \textit{How abelian is a finite group?},
Linear and Multilinear Algebra \textbf{3} (1976), 307-312.

\bibitem{1933Hall} P. Hall, \textit{A contribution to the theory of groups of prime power order},
 Proceedings of the London Mathematical Society \textbf{36} (1933), 29-95.

\bibitem{1991Heineken} H. Heineken, \textit{Nilpotent subgroups of finite solvable groups},
Archiv der Mathematik \textbf{56} (1991), 417-423.

\bibitem{2018Kourovka} E.I. Khukhro, V.D. Mazurov, \textit{Unsolved Problems in Group Theory: The Kourovka Notebook}
\textbf{19}, Sobolev Institute of Mathematics (2018).

\bibitem{1978O} A.Y. Ol'shanskii, \textit{The number of generators and orders of abelian subgroups of finite $p$-groups},
Mathematical Notes \textbf{23} (1978), 183-185.

\bibitem{1997Pyber} L. Pyber, \textit{How abelian is a finite group?},
The Mathematics of Paul Erd\H{o}s,
Algorithms and Combinatorics \textbf{13} (1997 Springer, Berlin), 372-384.

\bibitem{1969Thompson} J. Thompson,
   \textit{A replacement theorem for $p$-groups and a conjecture},
	Journal of Algebra \textbf{13} (1969), 149-151.

  \vspace{1cm}

\end{document}